\documentclass[11pt]{amsart}

\usepackage{amssymb}
\usepackage{amsthm}
\usepackage{amsmath}
\usepackage{commath}
\usepackage{amsbsy}
\usepackage[all]{xy}
\usepackage{bm}
\usepackage{hyperref}
\usepackage{tikz}
\usepackage{array}
\usepackage{colortbl,xcolor}
\usepackage{ytableau}
\usepackage{hhline}
\usepackage{graphicx, wrapfig}
\usepackage[margin=1in]{geometry}
\usepackage{listings}
\usepackage{listings}
\usepackage{colortbl}
\usepackage{xcolor}
\usepackage{enumitem}
\usepackage{mathtools}

\allowdisplaybreaks
\newtheorem{theorem}{Theorem}[section] 
\newtheorem{proposition}[theorem]{Proposition}
\newtheorem{lemma}[theorem]{Lemma} 

\newtheorem{corollary}[theorem]{Corollary}

\theoremstyle{definition}
\newtheorem{definition}[theorem]{Definition}

\newtheorem{conjecture}[theorem]{Conjecture}
\newtheorem{exm}[theorem]{Example}

\newtheorem{rem}[theorem]{Remark}

\newcommand{\svdots}{\raisebox{3pt}{$\scalebox{.75}{\vdots}$}}

\title[Ehrhart-Equivalence of integral $n$-polytopes in $\mathbb{R}^n$]{Ehrhart-Equivalence, Equidecomposability, and Unimodular Equivalence of Integral Polytopes}

\author[F. Abney-McPeek]{Fiona Abney-McPeek}
\address{University of Chicago Laboratory Schools, Chicago, IL 60637}
\email{fabney-mcpeek@ucls.uchicago.edu, f.abneymcp@gmail.com}

\author[S. Biswas]{Sanket Biswas}
\address{Sant Longowal Institute of Engineering and Technology, Sangrur, Punjab-148106, India}
\email{1940108@sliet.ac.in, snktbiswas@gmail.com}

\author[S. Dutta]{Senjuti Dutta}
\address{Chennai Mathematical Institute, Kelambakkam-603103, India}
\email{senjuti@cmi.ac.in, senjutid58@gmail.com}

\author[Y. Huang]{Yongyuan Huang}
\address{Boston University, Boston, MA 02215}
\email{yhuang22@bu.edu}

\author[D. Li]{Deyuan Li}
\address{Yale University, New Haven, CT 06511}
\email{deyuan.li@yale.edu}

\author[N. Xu]{Nancy Xu}
\address{Princeton University, Princeton, NJ 08544}
\email{nancyx@princeton.edu, xu.nancy26@gmail.com}

\begin{document}

\begin{abstract}
    Ehrhart polynomials are extensively-studied structures that interpolate the discrete volume of the dilations of integral $n$-polytopes. The coefficients of Ehrhart polynomials, however, are still not fully understood, and it is not known when two polytopes have equivalent Ehrhart polynomials. In this paper, we establish a relationship between Ehrhart-equivalence and other forms of equivalence: the $\operatorname{GL}_n(\mathbb{Z})$-equidecomposability and unimodular equivalence of two integral $n$-polytopes in $\mathbb{R}^n$. We conjecture that any two Ehrhart-equivalent integral $n$-polytopes $P,Q\subset\mathbb{R}^n$ are $\operatorname{GL}_n(\mathbb{Z})$-equidecomposable into $\frac{1}{(n-1)!}$-th unimodular simplices, thereby generalizing the known cases of $n=1, 2, 3$. We also create an algorithm to check for unimodular equivalence of any two integral $n$-simplices in $\mathbb{R}^n$. We then find and prove a new one-to-one correspondence between unimodular equivalence of integral $2$-simplices and the unimodular equivalence of their $n$-dimensional pyramids. Finally, we prove the existence of integral $n$-simplices in $\mathbb{R}^n$ that are not unimodularly equivalent for all $n \ge 2$.
\end{abstract}

\maketitle
\section{Introduction}
\subsection{Background}
Pick's Theorem, described by Georg Pick in 1899 \cite{pick}, states that for any simple polygon in $\mathbb{R}^2$ whose vertices are lattice points, the area of the polygon is given by $$A=I + \frac{B}{2} - 1,$$ where $I$ is the number of interior lattice points and $B$ is the number of lattice points on its boundary.

While Pick's Theorem holds for all such simple polygons in $\mathbb{R}^2$, Reeve proved that no generalized formula for Pick's Theorem exists in higher dimensions \cite{reeve}. In the 1960's, Eugène Ehrhart \cite{ehrhart} introduced the theory of Ehrhart polynomials as a generalization of Pick's Theorem for any convex rational $n$-polytopes\footnote{Unless otherwise stated, we shall use the term ``polytope'' in place of ``convex polytope'' throughout the rest of the paper.}, the convex hull of finitely many points in $\mathbb{Q}^n$. Ehrhart proved that, for any nonnegative integer $t$ and any rational $n$-polytope $P \subset \mathbb{R}^n$, the number of lattice points in $tP$ can be interpolated by a quasi-polynomial $L_P(t)$ of degree $n$, so that for all $t \in \mathbb{Z}_{>0}$,
$$L_P(t)=|tP\cap\mathbb{Z}^n|=c_n(t)t^n +\cdots+c_1(t)t+c_0(t).$$ 
$L_p$ is called the \emph{Ehrhart quasi-polynomial} of $P$. It is known that:
\begin{enumerate}
    \item The coefficients $c_0(t), \ldots, c_n(t)$ are periodic functions from $\mathbb{Z}$ to $\mathbb{Q}$.
    \item Evaluating $L_P$ at negative integers yields
    $$L_P(-t)=(-1)^{\dim P}L_{P^{\circ}} (t),$$ where $P^{\circ}$ denotes the strict interior of the polytope $P$ and $$L_{P^{\circ}} (t)=|tP^{\circ}\cap\mathbb{Z}^n|.$$ This is known as the \emph{Ehrhart-Macdonald reciprocity law}. 
\end{enumerate}

When $P$ is an integral polytope, $L_P(t)$ must be an $n$-dimensional polynomial satisfying:
\begin{enumerate}
    \item the coefficients $c_0, \hdots, c_n$ are rational constants.
    \item $c_0$ is the Euler characteristic of $P$, and so $c_0=1$ for closed convex polytopes.
    \item $c_n$ is the Euclidean volume of $P$.
\end{enumerate} In this case we call $L_P$ to be the Ehrhart polynomial of $P$. 

While the Ehrhart polynomial of integral $n$-polytopes relates its Euclidean volume to the number of lattice points both strictly within the polytope and on its boundary, thereby acting as a generalization of Pick's Theorem, a lot is still unknown about the behavior of Ehrhart polynomials. In this paper, we establish a relationship between Ehrhart-equivalence and $\operatorname{GL}_n(\mathbb{Z})$-equidecomposability, and unimodular equivalence of two integral $n$-polytopes in $\mathbb{R}^n$, two other forms of equivalence.

\subsection{General Definitions}
We first define the concepts of Ehrhart-equivalence, unimodular equivalence and $\operatorname{GL}_n(\mathbb{Z})$-equidecomposability.

\begin{definition}
We call two rational polytopes $Q_1,Q_2\subset\mathbb{R}^n$ \emph{Ehrhart-equivalent} if for all $k \in \mathbb{Z}_{\ge 1}$, $|kQ_1\cap \mathbb{Z}^n|=|kQ_2\cap \mathbb{Z}^n|$. This occurs exactly when $Q_1$ and $Q_2$ have the same Ehrhart quasi-polynomial.
\end{definition}

\begin{exm}
Consider the $4$-dimensional integral polytopes $P_1$ and $P_2$ defined by $$P_1= \mathrm{conv}\{(0,0,0,0),(1,1,1,1),(1,1,2,2),(3,2,1,2),(2,3,1,3)\}$$
and $$P_2= \mathrm{conv}\{(3,3,3,3),(4,4,3,3),(7,6,9,10),(1,2,1,1),(1,1,2,1)\}.$$ It turns out that $P_1$ and $P_2$ have the same Ehrhart polynomial
$$L_{P_1}(t)=\frac{1}{8}t^4 + \frac{3}{4}t^3 + \frac{15}{8}t^2 + \frac{9}{4}t + 1=L_{P_2}(t).$$ Hence, $P_1$ and $P_2$ are Ehrhart-equivalent. 
\end{exm} 

However, other forms of equivalences also exist for $n$-polytopes, and we show that they are closely related to Ehrhart-equivalence.

\begin{definition}
An \emph{affine-unimodular transformation} $U: \mathbb{R}^n \to \mathbb{R}^n$ is a transformation of the form
$$U(\mathbf{v}) = A\mathbf{v}+\mathbf{b},$$
where $A\in \operatorname{GL}_n(\mathbb{Z})$ and $\mathbf{b}\in\mathbb{Z}^n$. This is denoted by $U \in \operatorname{GL}_n(\mathbb{Z})\ltimes\mathbb{Z}^n$.
\end{definition}

\begin{definition} \label{uni eqi} We call two polytopes $P,Q\subset\mathbb{R}^n$ \emph{unimodularly equivalent} if there exists some affine-unimodular transformation $U\in\operatorname{GL}_n(\mathbb{Z})\ltimes\mathbb{Z}^n$ satisfying $U(P)=Q$.\end{definition}

Note that unimodular equivalence is an equivalence relation, and it turns out to be surprisingly closely related to known results on Ehrhart-equivalence. We call a set $S \subset \mathbb{R}^n$ \emph{relatively open} if it is open inside its affine span \cite{3D}. This leads us to another notion of equivalence.
\begin{definition}
We say that two polytopes $P, Q \subset \mathbb{R}^n$ are \emph{$\operatorname{GL}_n(\mathbb{Z})$-equidecomposable} if there exist relatively open simplices $P_1,\dots,P_r$ and $Q_1,\dots,Q_r$ with $$P=\bigsqcup_{i=1}^r P_i\quad \text{and}\quad Q=\bigsqcup_{i=1}^r Q_i$$ satisfying for all $1 \le i \le r$, $U_i(P_i)=Q_i$ and $U_i\in \operatorname{GL}_n(\mathbb{Z})\ltimes \mathbb{Z}^n$ are affine-unimodular transformations.
\end{definition}

\subsection{Main Results}
The rest of the paper is formatted as follows. We begin in Section~\ref{prelim} with relevant terminology and past results on the relationships among Ehrhart-equivalence, $\operatorname{GL}_n(\mathbb{Z})$-equidecomposability, and unimodular equivalence. In Section~\ref{main results}, we introduce a generalized conjecture that any two Ehrhart-equivalent integral $n$-polytopes are $\operatorname{GL}_n(\mathbb{Z})$-equidecomposable into $\tfrac{1}{(n-1)!}$th-unimodular simplices.  We also provide an algorithm to test for whether two $n$-simplices are unimodularly equivalent and proved a generalized result of unimodular equivalence given Ehrhart-equivalent $n$-polytopes. Then in Section~\ref{pyramids}, we establish a one-to-one correspondence between the unimodular equivalence of $2$-simplices and unimodular equivalence of their $n$-dimensional pyramids. Finally in Section~\ref{future}, we raise several conjectures and open questions and suggest avenues for further investigation.

\section{Preliminaries} \label{prelim}
\begin{rem}\label{uni}If two integral simplices $S,T\subset\mathbb{R}^n$ are unimodularly equivalent, then they are also $\operatorname{GL}_n(\mathbb{Z})$-equidecomposable by definition.\end{rem}

\begin{proposition}\label{dilations}
If two polytopes $P,Q\subset\mathbb{R}^n$ are $\operatorname{GL}_n(\mathbb{Z})$-equidecomposable, then for all $k \in \mathbb{N}$, $kP$ and $kQ$ are also $\operatorname{GL}_n(\mathbb{Z})$-equidecomposable.
\end{proposition}

\begin{proof}
Consider any two polytopes $P,Q\subset\mathbb{R}^n$ that are $\operatorname{GL}_n(\mathbb{Z})$-equidecomposable. This by definition means that there exists relatively open simplices $P_1,\dots,P_r$ and $Q_1,\dots,Q_r$ such that $U_i(P_i)=Q_i$, for all $1\le i\le r$, where $U_i\in \operatorname{GL}_n(\mathbb{Z})\ltimes \mathbb{Z}^n$, $1\le i\le r$ are affine-unimodular transformations and $$P=\bigsqcup_{i=1}^r P_i\quad \text{and}\quad Q=\bigsqcup_{i=1}^r Q_i.$$

If $U_i$ is of the form
\begin{align*}
    U_i(\mathbf{x})= A_i\mathbf{x}+\mathbf{c_i}
\end{align*}
for $x, y \in \mathbb{R}^n$, then for any $k \in \mathbb{N}$
$$kU_i(\mathbf{x}) =kA_i\mathbf{x}+k\mathbf{c_i} = A_i(k\mathbf{x})+k\mathbf{c_i},$$
so $kP_i$ and $kQ_i$ are unimodularly equivalent via the affine-unimodular transformation $V_i\in\operatorname{GL}_n(\mathbb{Z})\ltimes \mathbb{Z}^n$.
where $$V_i(\mathbf{v})=  A_i\mathbf{v}+ k\mathbf{c_i},$$ for all $\mathbf{v}\in\mathbb{R}^n.$ Now, for any $k\in\mathbb{N}$ we have 
\begin{align*}
    kP &= k\bigsqcup^r_{i = 1}P_i = \bigsqcup^r_{i = 1}kP_i, \\
    kQ &= k\bigsqcup^r_{i = 1}Q_i = \bigsqcup^r_{i = 1}kQ_i.
\end{align*} Hence, by definition $kP$ and $kQ$ are $\operatorname{GL}_n(\mathbb{Z})$-equidecomposable for all $k\in\mathbb{N}.$
\end{proof}

\begin{theorem}\label{equivalence}If two rational polytopes $P,Q\subset\mathbb{R}^n$ are $\operatorname{GL}_n(\mathbb{Z})$-equidecomposable, then they are Ehrhart-equivalent.\end{theorem}

\begin{proof} 
Suppose $P, Q \subset \mathbb{R}^n$ are rational polytopes that are $\operatorname{GL}_n(\mathbb{Z})$-equidecomposable. Thus, there are relatively open simplices $P_1, \dots, P_r$ and $Q_1,\dots, Q_r$ such that $U_i(P_i)=Q_i$,  where $U_i\in \operatorname{GL}_n(\mathbb{Z})\ltimes \mathbb{Z}^n$ are affine-unimodular transformations for all $1\le i\le r$  and $$P=\bigsqcup_{i=1}^r P_i\quad \text{and}\quad Q=\bigsqcup_{i=1}^r Q_i.$$
This yields a bijection between the lattice points in each $P_i$ and $Q_i$, thereby inducing a bijection between the lattice points in $P$ and $Q$. This implies by Proposition~\ref{dilations} there is also a bijection between the lattice points in $kP$ and $kQ$ for all $k$, resulting in
$$|kP \cap \mathbb{Z}^n| = |kQ \cap \mathbb{Z}^n|.$$ 
By definition, this implies that $P$ and $Q$ are Ehrhart-equivalent.\end{proof}

\begin{corollary}\label{uni2}If two integral simplices $S,T\subset\mathbb{R}^n$ are unimodularly equivalent, then they are Ehrhart-equivalent.\end{corollary}

\begin{proof}Select two arbitrary integral simplices $S,T\subset\mathbb{R}^n$ such that they are unimodularly equivalent. By Remark~\ref{uni}, $S$ and $T$ are $\operatorname{GL}_n(\mathbb{Z})$-equidecomposable. Thus, by Theorem~\ref{equivalence}, $S$ and $T$ are Ehrhart-equivalent.\end{proof}

While Thoerem~\ref{equivalence} tells us that Ehrhart-equvalent rational polytopes must also be $\operatorname{GL}_n(\mathbb{Z})$-equidecomposable, the converse may not necessarily hold. One approach to establishing weaker versions of the converse has been to use unimodular simplices.

\begin{definition}
An integral $n$-simplex $S\subset\mathbb{R}^n$ is said to be \emph{unimodular} if it has normalized volume $1$.
\end{definition}

\begin{proposition} \label{unisimp}
If $S, T \subset\mathbb{R}^n$ are unimodular $n$-simplices, then $S$ and $T$ are unimodularly equivalent.
\end{proposition}
This can be proven by expressing any unimodular simplex with a vertex at the origin as an integer matrix with determinant $\pm 1$ and using the fact that $\operatorname{GL_n}(\mathbb{Z})$ is a group (see for instance \cite{4-simplices})

In \cite{3D}, Erbe, Haase, and Santos proved the converse of Theorem~\ref{equivalence} for integral $2$-polytopes by using the fact that all integral polygons have unimodular triangulations. The proof, however, fails for higher dimensions $n\ge 3$, since not all integral $n$-polytopes have unimodular triangulations. We build up the background to prove a weaker version of the converse of Theorem~\ref{equivalence}.

From this proposition, we later develop Proposition~\ref{unisimp1} to check the $\operatorname{GL}_n(\mathbb{Z})$-equidecomposability of two Ehrhart-equivalent integral polytopes $P,Q\subset\mathbb{R}^n$. Proposition~\ref{unisimp} also invokes the question of when an integral polytope $P\subset\mathbb{R}^n$ can be unimodularly triangulated, which leads us to the following theorem:

\begin{theorem}[\cite{weirdo}] \label{kP} For every integral polytope $P$, there is a constant $k$ such that the dilation $kP$ admits a unimodular triangulation.\end{theorem}

\begin{definition} \label{def: weakly}
Two rational polytopes $P,Q\subset \mathbb{R}^n$ are \emph{weakly $\operatorname{GL}_n(\mathbb{Z})$-equidecomposable} if they can be decomposed into rational polytopes $P_1, P_2, \dots, P_k$ and $Q_1, Q_2, \dots, Q_k$, so that for all $1 \le i \le k$, $P_i$ and $Q_i$ are unimodularly equivalent via $\operatorname{GL}_n{(\mathbb{Z})} \ltimes \mathbb{Q}^n.$\end{definition}

This definition is equivalent to stating that $P$ and $Q$ are weakly $\operatorname{GL}_n(\mathbb{Z})$-equidecomposable if there is a dilation factor $k\in\mathbb{Z}_{>0}$ such that $kP$ and $kQ$ are $\operatorname{GL}_n(\mathbb{Z})$-equidecomposable \cite{weakly}.

\begin{theorem}[\cite{weakly}] \label{poly} If a polytope $P\subset\mathbb{R}^n$ can be unimodularly triangulated, then for all $k \in \mathbb{N}$, $kP$ can also be unimodularly triangulated.\end{theorem}

This theorem, combined with Definition~\ref{def: weakly} together yield the following result:  

\begin{theorem}[\cite{weakly}] \label{weakly}
Two integral polytopes $P,Q\subset\mathbb{R}^n$ are weakly $\operatorname{GL}_n(\mathbb{Z})$-equidecomposable if and only if they are Ehrhart-equivalent.
\end{theorem}

\begin{corollary}\label{cor1}If $P,Q\subset \mathbb{R}^n$ are two arbitrary Ehrhart-equivalent integral polytopes, then there are infinitely many $k\in\mathbb{Z}_{\ge 1}$ such that $kP$ and $kQ$ are $\operatorname{GL}_n(\mathbb{Z})$-equidecomposable.
\end{corollary}

\begin{proof} Consider any two Ehrhart-equivalent integral polytopes $P,Q\subset\mathbb{R}^n$. By Theorem~\ref{weakly} and Definition~\ref{def: weakly}, we conclude that there exists some $d\in\mathbb{Z}_{\ge 1}$ such that $dP$ is $\operatorname{GL}_n(\mathbb{Z})$-equidecomposable to $dQ$. This is turn by Proposition~\ref{dilations} implies that, $jdP$ is $\operatorname{GL}_n(\mathbb{Z})$-equidecomposable to $jdQ$ for all $j\in\mathbb{Z}_{\ge 1}$.\end{proof}

Note that Theorem~\ref{weakly} serves as a weaker version of the converse of Theorem~\ref{equivalence}, since, it implies that, given two arbitrary Ehrhart-equivalent integral polytopes $P,Q\subset\mathbb{R}^n$, there is always some dilation factor $k\in\mathbb{Z}_{\ge 1}$, such that $kP$ and $kQ$ are $\operatorname{GL}_n(\mathbb{Z})$-equidecomposable. 

Now, for any $k\in\mathbb{Z}_{\ge 1}$, define a $\frac{1}{k}$-th unimodular simplex to be a simplex that is unimodular when dilated by a factor of $k$. In particular, an $n$-simplex $S\subset\mathbb{R}^n$ is defined to be $\frac{1}{k}$-th unimodular if the normalized volume of $kS$ is $1$. 

Using Theorem~\ref{kP} and Theorem~\ref{weakly}, Erbe, Haase, and Santos, proved the following theorem for integral $3$-polytopes in $\mathbb{R}^3$:

\begin{theorem}[\cite{3D}] \label{thm: 3D}
If two integral $3$-polytopes $P,Q\subset\mathbb{R}^3$ are Ehrhart-equivalent, then $P$ and $Q$ are $\operatorname{GL}_3(\mathbb{Z})$-equidecomposable into half-unimodular simplices. This is equivalent to stating that if two integral $3$-polytopes $P,Q\subset\mathbb{R}^3$ are Ehrhart-equivalent, then $2P$ and $2Q$ are $\operatorname{GL}_3(\mathbb{Z})$-equidecomposable.\end{theorem} 

\begin{corollary}\label{cor2} If $P,Q\subset\mathbb{R}^3$ are two arbitrary Ehrhart-equivalent integral $3$-polytopes, then $2jP$ and $2jQ$ are $\operatorname{GL}_3(\mathbb{Z})$-equidecomposable for all $j\in\mathbb{Z}_{\ge 1}$.\end{corollary}

\begin{proof}Let $P, Q \subset\mathbb{R}^3$ be any two integral $3$-polytopes that are Ehrhart-equivalent. By Theorem~\ref{thm: 3D}, $2P$ and $2Q$ are $\operatorname{GL}_3(\mathbb{Z})$-equidecomposable. Thus, by Proposition~\ref{dilations} we conclude that $2jP$ and $2jQ$ are $\operatorname{GL}_3(\mathbb{Z})$-equidecomposable for all $j\in\mathbb{Z}_{\ge 1}$.\end{proof}

\section{Relations between equivalences} \label{main results}

\subsection{$\operatorname{GL}_n(\mathbb{Z})$-equidecomposability of Ehrhart-equivalent integral polytopes}

By Proposition~\ref{unisimp}, one way to prove that two integral polytopes $P, Q \subset \mathbb{R}^n$ are $\operatorname{GL}_n(\mathbb{Z})$-equidecomposability is to analyze at their $n$-dimensional unimodular triangulations, provided they exist. Based on the this approach, the following proposition gives us a way of testing when two Ehrhart-equivalent integral polytopes in $\mathbb{R}^n$ are $\operatorname{GL}_n(\mathbb{Z})$-equidecomposable. 

\begin{proposition}\label{unisimp1} If $P,Q\subset\mathbb{R}^n$ are two Ehrhart-equivalent integral $n$-polytopes such that $$P=\bigsqcup_{i=1}^p P_i\quad \text{and}\quad Q=\bigsqcup_{i=1}^q Q_i,$$ where $P_1,P_2,\dots, P_p$ and $Q_1,Q_2,\dots, Q_q$ are relatively open unimodular $n$-simplices, then $P$ and $Q$ are $\operatorname{GL}_n(\mathbb{Z})$-equidecomposable.\end{proposition}

\begin{proof}Since $P$ and $Q$ are Ehrhart-equivalent, we have that $$L_P(t)=a_nt^n+\cdots+a_1t+1=L_Q(t),$$ for some $a_1, \ldots, a_n \in \mathbb{Q}$. This implies that $$\mathrm{vol}(P)=a_n=\mathrm{vol}(Q).$$ Now since $P_1,P_2,\dots, P_p$ and $Q_1,Q_2,\dots, Q_q$ are unimodular $n$-simplices in $\mathbb{R}^n$, we have $$\mathrm{vol}(P_j)=\frac{1}{n!}=\mathrm{vol}(Q_k),$$ for all $1\le j\le p$ and $1\le k\le q$. Since $P_1,P_2,\dots, P_p$ and $Q_1,Q_2,\dots, Q_q$ are relatively open simplices with $$P=\bigsqcup_{i=1}^p P_i\quad \text{and}\quad Q=\bigsqcup_{i=1}^q Q_i,$$ we have $$\frac{p}{n!}=\mathrm{vol}(P)=\mathrm{vol}(Q)=\frac{q}{n!},$$ which in turn implies that $p=q$. Let $p=q=r$. Then $$P=\bigsqcup_{i=1}^r P_i\quad \text{and}\quad Q=\bigsqcup_{i=1}^r Q_i.$$ Finally, by Proposition~\ref{unisimp}, for any $1\le j\le r$, $P_j$ is unimodularly equivalent to $Q_k$ for all $1\le k\le r$. As a result, $P_i$ is unimodularly equivalent to $Q_i$ for all $1\le i\le r$. Then for all $1\le i\le r$, there exists affine-unimodular transformations $U_i\in\operatorname{GL}_n(\mathbb{Z})\ltimes \mathbb{Z}^n$, such that $U_i(P_i)=Q_i.$ Thus $P$ and $Q$ are $\operatorname{GL}_n(\mathbb{Z})$-equidecomposable, as desired.\end{proof} 

We know by Corollary~\ref{cor1} that if two integral polytopes $P,Q\subset\mathbb{R}^n$ are Ehrhart-equivalent, then there exists infinitely many $k\in\mathbb{Z}_{\ge 1}$ such that $kP$ and $kQ$ are $\operatorname{GL}_n(\mathbb{Z})$-equidecomposable. We now attempt to find structures which characterize the dilation factors $k$ for which $kP$ and $kQ$ are $\operatorname{GL}_n(\mathbb{Z})$-equidecomposable. Note by Theorem~\ref{thm: 3D} that for $n=1,2,3$, the smallest such values of $k$ are $k=1,1,2$, respectively. We conjecture a generalization that for all $n \in \mathbb{Z}_{\ge 1}$, $(n-1)!$ is one such value of $k$, though it may not necessarily be the smallest. 

\begin{conjecture}\label{ked}
For any $n \in \mathbb{Z}_{\ge 1}$, if $P,Q\subset\mathbb{R}^n$ are two arbitrary Ehrhart-equivalent integral $n$-polytopes, then they are $\operatorname{GL}_n(\mathbb{Z})$-equidecomposable into $\frac{1}{(n-1)!}$th-unimodular simplices. In other words, if $P,Q\subset\mathbb{R}^n$ are two arbitrary Ehrhart-equivalent integral $n$-polytopes, then $((n-1)!)P$ and $((n-1)!)Q$ are $\operatorname{GL}_n(\mathbb{Z})$-equidecomposable. 
\end{conjecture}

While Conjecture~\ref{ked} does not address the full characterization of all dilation factors $k$, it provides one such value $k$ for all $n\in\mathbb{Z}_{\ge 1}$. From Corollary~\ref{cor1}, we know that there are infinitely many values of $k$, and for $n=1$ and $n=2$, we know that any $k\in\mathbb{Z}_{\ge 1}$ works. We attempt to generalize an infinite class of such values $k$ for $n \ge 3$.  

\begin{conjecture}\label{gen} For any $n\in\mathbb{Z}_{>3}$, if $P,Q\subset\mathbb{R}^n$ are two arbitrary Ehrhart-equivalent integral $n$-polytopes, then $(j(n-1)!)P$ and $(j(n-1)!)Q$ are $\operatorname{GL}_n(\mathbb{Z})$-equidecomposable for all $j\in\mathbb{Z}_{\ge 1}$.\end{conjecture}

\begin{proof}[Proof of Conjecture~\ref{gen} from Conjecture~\ref{ked}] Let $P,Q\subset\mathbb{R}^n$ be two integral $n$-polytopes that are Ehrhart-equivalent. By Conjecture~\ref{ked}, $((n-1)!)P$ and $((n-1)!)Q$ are $\operatorname{GL}_n(\mathbb{Z})$-equidecomposable. Thus, by Proposition~\ref{dilations} we conclude that $(j(n-1)!)P$ and $(j(n-1)!)Q$ are $\operatorname{GL}_n(\mathbb{Z})$-equidecomposable for all $j\in\mathbb{Z}_{\ge 1}$.\end{proof}

In the following subsections, we build the necessary algorithms and background for addressing Conjecture~\ref{ked} for the specific case of $n=4$.

\subsection{Algorithm to generate Ehrhart-equivalent integral $n$-polytopes in $\mathbb{R}^n$}
We first describe an algorithm to generate Ehrhart-equivalent integral polytopes. We take a vertex set of arbitrary length and map the corresponding integral polytope to its Ehrhart polynomial. The hash collisions will give us different polytopes with the same Ehrhart polynomial, and hence they will be Ehrhart-equivalent.

\textbf{The algorithm is as follows:}
\begin{itemize}
\item Select an arbitrary finite vertex set $V= \{\mathbf{v_1},\mathbf{v_2},\dots,\mathbf{v_{m+1}}\}\subset\mathbb{Z}^n,$ such that $m\ge n$.

\item Obtain the integral $n$-polytope $P\subset\mathbb{R}^n$, corresponding to the vertex set $V$, defined by $$P=\mathrm{conv}(V),$$ using the SageMath code $$\texttt{sage: P = Polyhedron(vertices=[v\_1,v\_2,...,v\_(m+1)])}.$$  

\item Find the Ehrhart polynomial of $P$, denoted by $L_P(t)$, using the SageMath code $$\texttt{sage: P.ehrhart\_polynomial()}.$$

\item Make minor modifications on the vector set $V$ to obtain new vertex sets, then finding ``hash collisions'' among the vertex sets such that the corresponding polytope of a new vertex set $U$, given by $Q= \mathrm{conv}(U)(\neq P)$, has the same Ehrhart polynomial as that of $P$, using the SageMath code $$\texttt{sage: Q.ehrhart\_polynomial()}.$$
\end{itemize}

We implement this algorithm for $n=4$, and obtain the following sets of examples of Ehrhart-equivalent integral polytopes (described in a SageMath ready format). 

\textbf{Example 1. Simplices with Ehrhart polynomial: $\frac{1}{8}t^4 + \frac{3}{4}t^3 +\frac{15}{8}t^2 + \frac{9}{4}t + 1$}

\begin{center}
\begin{tabular}{| m{16 cm} |} 
\hline
\texttt{P\_1 = Polyhedron(vertices=[(0,0,0,0),(1,1,1,1),(1,1,2,2),(3,2,1,2),(2,3,1,3)])}\\ 
\hline
\texttt{P\_2 = Polyhedron(vertices=[(0,0,0,0),(2,1,1,1),(1,1,2,3),(2,2,1,2),(2,2,1,3)])}\\ 
\hline
\texttt{P\_3 = Polyhedron(vertices=[(0,0,0,1),(3,2,1,1),(1,1,2,3),(2,2,1,2),(2,2,1,3)])}\\ 
\hline
\texttt{P\_4 = Polyhedron(vertices=[(0,0,0,1),(3,2,2,1),(1,1,2,3),(2,2,1,2),(2,2,1,3)])}\\
\hline
\texttt{P\_5 = Polyhedron(vertices=[(0,0,0,1),(3,2,2,2),(1,1,2,3),(2,2,1,2),(2,2,1,3)])}\\
\hline
\texttt{P\_6 = Polyhedron(vertices=[(0,0,0,1),(3,2,2,3),(1,1,2,3),(2,2,1,2),(2,2,1,3)])}\\
\hline
\texttt{P\_7 = Polyhedron(vertices=[(0,0,0,1),(3,2,2,0),(1,1,2,3),(2,2,1,2),(2,2,1,3)])}\\
\hline
\texttt{P\_8 = Polyhedron(vertices=[(0,0,0,0),(2,2,2,1),(3,2,2,3),(1,1,1,2),(0,3,2,3)])}\\
\hline
\texttt{P\_9 = Polyhedron(vertices=[(3,3,3,3),(4,4,3,3),(7,6,9,10),(1,2,1,1),(1,1,2,1)])}\\
\hline
\end{tabular}
\end{center}
\vspace{2 mm}

\textbf{Example 2. Simplices with Ehrhart polynomial: $\frac{1}{24}t^4 + \frac{5}{12}t^3 + \frac{35}{24}t^2 + \frac{25}{12}t + 1$}
\begin{center}
\begin{tabular}{| m{16 cm} |} 
\hline
\texttt{P\_1 = Polyhedron(vertices=[(1,1,1,1),(2,3,1,1),(2,1,2,3),(2,2,1,2),(2,2,1,3)])}\\
\hline 
\texttt{P\_2 = Polyhedron(vertices=[(1,1,1,1),(2,3,1,2),(2,1,2,3),(2,2,1,2),(2,2,1,3)])]}\\
\hline
\texttt{P\_3 = Polyhedron(vertices=[(1,1,1,1),(2,3,1,1000),(2,1,2,3),(2,2,1,2),
(2,2,1,3)])}\\
\hline
\texttt{P\_4 = Polyhedron(vertices=[(0,0,0,0),(1,2,2,1),(3,1,2,3),(1,1,1,2),(0,3,2,3)])}\\
\hline
\texttt{P\_5 = Polyhedron(vertices=[(0,0,0,0),(1,3,2,1),(3,2,2,3),(1,1,1,2),(0,3,2,3)])}\\
\hline
\texttt{P\_6 = Polyhedron(vertices=[(0,0,0,0),(1,2,2,4),(3,2,2,3),(1,1,1,2),(0,3,2,3)])}\\
\hline
\texttt{P\_7 = Polyhedron(vertices=[(0,0,0,0),(1,0,0,0),(0,1,0,0),(0,0,1,0),(0,0,0,1)])}\\
\hline
\texttt{P\_8 = Polyhedron(vertices=[(5,7,6,7),(6,5,7,5),(5,7,6,6),(6,5,6,7),(6,6,6,6)])}\\
\hline
\texttt{P\_9 = Polyhedron(vertices=[(5,7,6,7),(6,7,7,5),(5,7,6,6),(6,5,6,7),(6,6,6,6)])}\\
\hline
\texttt{P\_10 = Polyhedron(vertices=[(6,8,7,8),(7,8,8,6),(6,8,7,7),(7,6,7,8),(7,7,7,7)])}\\
\hline
\texttt{P\_11 = Polyhedron(vertices=[(10,12,11,12),(11,12,12,10),(10,12,11,11),
(11,10,11,12),(11,11,11,11)])}\\
\hline 
\texttt{P\_12 = Polyhedron(vertices=[(1010,1012,1011,1012),(1011,1012,1012,1010),
(1010,1012,1011,1011),(1011,1010,1011,1012),(1011,1011,1011,1011)])}\\
\hline 
\texttt{P\_13 = Polyhedron(vertices=[(0,0,0,0),(1,0,0,0),(1,1,1,1),(0,0,1,0),(0,1,0,0)])}\\
\hline
\texttt{P\_14 = Polyhedron(vertices=[(3,1,1,1),(3,2,2,2),(2,2,2,2),(1,2,1,1),(1,1,2,1)])}\\
\hline 
\texttt{P\_15 = Polyhedron(vertices=[(3,3,1,1),(3,2,2,2),(2,2,2,2),(1,2,1,1),(1,1,2,1)])}\\
\hline 
\texttt{P\_16 = Polyhedron(vertices=[(3,3,3,3),(3,2,2,2),(2,2,2,2),(1,2,1,1),(1,1,2,1)])}\\
\hline 
\end{tabular}
\end{center}
\begin{center}
\begin{tabular}{| m{16 cm} |} 
\hline
\texttt{P\_17 = Polyhedron(vertices=[(3,3,3,3),(4,3,3,3),(2,2,2,2),(1,2,1,1),(1,1,2,1)])}\\
\hline 
\texttt{P\_18 = Polyhedron(vertices=[(3,3,3,3),(4,3,3,3),(3,2,2,2),(1,2,1,1),(1,1,2,1)])}\\
\hline 
\texttt{P\_19 = Polyhedron(vertices=[(3,3,3,3),(4,3,3,3),(3,3,2,2),(1,2,1,1),(1,1,2,1)])}\\
\hline 
\texttt{P\_20 = Polyhedron(vertices=[(3,3,3,3),(4,4,3,3),(4,3,3,2),(1,2,1,1),(1,1,2,1)])}\\
\hline 
\texttt{P\_21 = Polyhedron(vertices=[(3,3,3,3),(4,4,3,3),(7,6,3,2),(1,2,1,1),(1,1,2,1)])}\\
\hline 
\texttt{P\_22 = Polyhedron(vertices=[(1,1,1,1),(3,2,2,2),(2,2,2,2),(1,2,1,1),(1,1,2,1)])}\\
\hline
\end{tabular}
\end{center}
\vspace{2 mm}

\textbf{Example 3. Simplices with Ehrhart polynomial: $\frac{1}{6}t^4 + t^3 + \frac{7}{3}t^2 + \frac{5}{2}t + 1$}
\begin{center}
\begin{tabular}{| m{16 cm} |} 
\hline
\texttt{P\_1 = Polyhedron(vertices=[(0,0,0,0),(1,2,2,1),(3,2,2,3),(1,1,1,2),(0,3,2,3)])}\\
\hline
\texttt{P\_2 = Polyhedron(vertices=[(0,0,0,0),(1,2,2,1),(3,2,2,3),(1,1,1,2),(1,3,2,3)])}\\
\hline
\texttt{P\_3 = Polyhedron(vertices=[(0,0,0,0),(1,2,2,1),(3,2,2,3),(1,1,1,2),(2,3,2,3)])}\\
\hline
\texttt{P\_4 = Polyhedron(vertices=[(0,0,0,0),(1,2,2,1),(3,2,2,3),(1,1,1,2),(3,3,2,3)])}\\
\hline
\texttt{P\_5 = Polyhedron(vertices=[(0,0,0,0),(1,2,2,1),(3,2,2,3),(1,1,1,2),
(1000,3,2,3)])}\\
\hline
\texttt{P\_6 = Polyhedron(vertices=[(0,0,0,0),(1,2,2,1),(3,2,2,3),(1,1,1,2),(0,3,4,3)])}\\
\hline
\texttt{P\_7 = Polyhedron(vertices=[(0,0,0,0),(1,2,2,1),(3,2,2,3),(1,1,1,2),(0,3,2,1)])}\\
\hline
\texttt{P\_8 = Polyhedron(vertices=[(0,0,0,0),(1,2,2,1),(3,2,2,3),(1,1,1,2),(0,3,2,2)])}\\
\hline
\texttt{P\_9 = Polyhedron(vertices=[(0,0,0,0),(1,2,2,1),(3,2,2,3),(1,1,1,2),(0,3,2,4)])}\\
\hline
\texttt{P\_10 = Polyhedron(vertices=[(0,0,0,0),(1,2,2,1),(3,2,2,3),(1,1,1,2),
(0,3,2,1000)])}\\
\hline
\end{tabular}
\end{center}
\vspace{2 mm}

\textbf{Example 4. Simplices with Ehrhart polynomial: $\frac{1}{8}t^4 + \frac{5}{12}t^3 + \frac{11}{8}t^2 + \frac{25}{12}t + 1$}
\begin{center}
\begin{tabular}{| m{16 cm} |} 
\hline
\texttt{P\_1 = Polyhedron(vertices=[(0,0,0,1),(1,0,0,0),(1,1,1,1),(0,1,0,0),(0,0,1,0)])}\\
\hline
\texttt{P\_2 = Polyhedron(vertices=[(2,2,2,1),(1,0,0,0),(1,1,1,1),(0,1,0,0),(0,0,1,0)])}\\
\hline
\texttt{P\_3 = Polyhedron(vertices=[(3,3,3,1),(3,2,2,2),(2,2,2,2),(1,2,1,1),(1,1,2,1)])}\\
\hline
\texttt{P\_4 = Polyhedron(vertices=[(3,3,3,3),(4,3,3,3),(3,3,3,2),(1,2,1,1),(1,1,2,1)])}\\
\hline
\end{tabular}
\end{center}
\vspace{2 mm}

\textbf{Example 5. Simplices with Ehrhart polynomial: $8t^4 + \frac{38}{3}t^3 + 6t^2 + \frac{7}{3}t + 1$}
\begin{center}
\begin{tabular}{| m{16 cm} |} 
\hline
\texttt{P\_1 = Polyhedron(vertices=[(4,3,2,1),(6,7,5,8),(10,9,12,11),(13,14,15,16),
(18,19,17,20)])}\\
\hline
\texttt{P\_2 = Polyhedron(vertices=[(5,4,3,2),(7,8,6,9),(11,10,13,12),(14,15,16,17),
(19,20,18,21)])}\\
\hline
\texttt{P\_3 = Polyhedron(vertices=[(6,4,3,2),(8,8,6,9),(12,10,13,12),(15,15,16,17),
(20,20,18,21)])}\\
\hline
\end{tabular}
\end{center}
\vspace{2 mm}

\textbf{Example 6. Polytopes with Ehrhart polynomial: $\frac{1}{6}t^4 + \frac{2}{3}t^3 + \frac{11}{6}t^2 + \frac{7}{3}t + 1$}
\begin{center}
\begin{tabular}{| m{16 cm} |} 
\hline
\texttt{P\_1 = Polyhedron(vertices=[(0,0,0,0),(0,0,0,1),(0,0,1,0),(0,1,0,0),(1,0,0,0),
(1,1,1,1)])}\\
\hline
\texttt{P\_2 = Polyhedron(vertices=[(0,0,0,0),(1,2,2,4),(3,2,2,3),(1,1,1,2),(0,3,2,3),
(5,8,7,12)])}\\
\hline
\end{tabular}
\end{center}
In the next subsection, we demonstrate a SageMath code to determine whether two arbitrary integral $n$-simplices $S,T\subset\mathbb{R}^n$ are unimodularly equivalent.

\subsection{Algorithm to check unimodular equivalence}
To build an algorithm that tests for unimodular equivalence, we first introduce some preliminary definitions and properties.
\begin{definition} For any $n$-simplex $S\subset\mathbb{R}^n$, where $$S= \mathrm{conv}\left\{(s_{11},s_{21},\dots,s_{n1}),(s_{12},s_{22},\dots,s_{n2}),(s_{13},s_{23},\dots,s_{n3}),\dots,(s_{1,n+1},s_{2,n+1},\dots,s_{n,n+1})\right\},$$ the \emph{definition matrix} of $S$, denoted by $D_S$, is $$D_S= \begin{pmatrix}
s_{11} & s_{12} & \hdots & s_{1,n+1} \\
\vdots & \vdots & \ddots & \vdots \\
s_{n1} & s_{n2} & \hdots & s_{n,n+1} \\
1 & 1 & \hdots & 1 
\end{pmatrix}.$$\end{definition}

\begin{definition} For any matrix $M$, the \emph{column set} of $M$, denoted by $\sigma_M$, is the set of all column vectors of $M$. If $$M= \begin{pmatrix}\mathbf{c_1} & \mathbf{c_2} & \dots & \mathbf{c_m}\end{pmatrix},$$ then $$\sigma_M=\{\mathbf{c_1},\mathbf{c_2},\dots,\mathbf{c_m}\}.$$\end{definition}

Consider any two arbitrary integral $n$-simplices $S,T\subset\mathbb{R}^n$, where $$S= \mathrm{conv}\{(s_{11},s_{21},\hdots,s_{n1}),(s_{12},s_{22},\hdots,s_{n2}),\hdots,(s_{1(n + 1)},s_{2(n + 1)},\hdots,s_{n(n + 1)})\}$$ and  $$T= \mathrm{conv}\{(t_{11},t_{21},\hdots,t_{n1}),(t_{12},t_{22},\hdots,t_{n2}), \hdots ,(t_{1(n + 1)},t_{2(n + 1)}, \hdots,t_{n(n + 1)})\}.$$ 

For all $1 \le i \le n+1$, let $\mathbf{s_{i}} =  (s_{1i}, s_{2i}, \hdots, s_{ni})$, and define $\mathbf{t_{i}}$ similarly\footnote{For the rest of the paper we assume that any arbitrary $n$-simplex $P\subset\mathbb{R}^n$ will be defined in the same way as $S$ and $T$ are defined here, unless otherwise stated. $\mathbf{p_i}$ will also be defined similarly to $\mathbf{s_i}$ and $\mathbf{t_i}$.}. Now, to prove that $S$ and $T$ are unimodularly equivalent, it is enough to show that there exists some $B\in \operatorname{GL}_n(\mathbb{Z})$, $\mathbf{c}\in\mathbb{Z}^n$, and some bijection $f:\{1,2,\hdots, n + 1\}\to \{1,2,\hdots, n + 1\}$, such that 
\begin{align}
    B\mathbf{s_i} + \mathbf{c} = \mathbf{t_{f(i)}}, \label{eq*}
\end{align}
for all $1 \leq i \leq n + 1$. 

For any matrix $M$, let $\mathrm{Sym}(\sigma_M)$ denote the symmetric group of $\sigma_M$. 
\begin{rem}\label{sym} Showing \eqref{eq*} is equivalent to showing that there exists some $A\in\mathbb{M}_{n + 1}(\mathbb{Z})$ and $\phi\in \mathrm{Sym}(\sigma_{D_T})$, with $$A = \begin{pmatrix} B & \mathbf{c} \\ \mathbf{0} & 1 \end{pmatrix},$$ where $B \in \operatorname{GL}_n(\mathbb{Z})$, $\mathbf{c} \in \mathbb{Z}^n$, such that $$AS^* = T^*,$$ where 
\begin{align}
S^*= \begin{pmatrix} \mathbf{s_1} & \mathbf{s_2} & \dots & \mathbf{s_{n + 1}} \\ 1 & 1 & \dots & 1\end{pmatrix}\quad \text{and} \quad T^*= \begin{pmatrix}\phi\dbinom{\mathbf{t_1}}{1} & \phi\dbinom{\mathbf{t_2}}{1} & \dots & \phi\dbinom{\mathbf{t_{n + 1}}}{1}\end{pmatrix}. \label{eq**}
\end{align}
Having this, \eqref{eq*} is equivalent to showing that there exists some $A\in\operatorname{GL}_{n + 1}(\mathbb{Z})$ and $\phi\in \mathrm{Sym}(\sigma_{D_T})$ such that $$AS^* = T^*,$$ where $S^*$ and $T^*$ are defined as in \eqref{eq**}.
\end{rem}

This leads us to the following proposition:

\begin{proposition}\label{fordir} Let $S$ and $T$ be two arbitrary integral $n$-simplices of equal volume contained in $\mathbb{R}^n$. If there exists some $A\in\mathbb{M}_{n + 1}(\mathbb{Z})$ and $\phi\in \mathrm{Sym}(\sigma_{D_T})$, such that $$AS^* = T^*,$$ where $S^*$ and $T^*$ are defined as in \eqref{eq**}, then $S$ and $T$ are unimodularly equivalent.
\end{proposition}

\begin{proof}
Suppose that there exists some $A\in\mathbb{M}_{n + 1}(\mathbb{Z})$ and $\phi\in \mathrm{Sym}(\sigma_{D_T})$, such that $$AS^* = T^*,$$ where $S^*$ and $T^*$ are defined as in \eqref{eq**}, and that $$A=\begin{pmatrix}B & \mathbf{c}\\ \mathbf{v} & u\end{pmatrix},$$ where $B\in\mathbb{M}_n(\mathbb{Z}), \mathbf{c}\in\mathbb{Z}^n, \mathbf{v}\in\mathbb{M}_{1,n}(\mathbb{Z})$ and $u\in\mathbb{Z}$. 

Then for every $1\le i\le n+1$, we have
$$A\binom{\mathbf{s_i}}{1}=\begin{pmatrix}
   B & \mathbf{c} \\
   \mathbf{v} & u
   \end{pmatrix} 
   \binom{\mathbf{s_i}}{1} = \binom{B\mathbf{s_i} + c}{\mathbf{v}\mathbf{s_i} + u}=\dbinom{\mathbf{\phi(t_i)}}{1}.$$

Thus, for all $1\le i\le n+1$, we have $$\mathbf{vs_i}+u=1,$$ which implies that $\mathbf{v}=0$ and $u=1$. Since $S$ and $T$ have the same volume, this implies that $$\frac{1}{n!}|\det S^*| = \mathrm{vol}(S) = \mathrm{vol}(T) = \frac{1}{n!}|\det T^*|$$ and so $$|\det S^*|=|\det T^*|\implies \det S^* = \pm \det T^*.$$ Thus, $\det A=\pm 1$. By Remark~\ref{sym} we therefore conclude that $S$ and $T$ are unimodularly equivalent.
\end{proof}

\begin{proposition}\label{revdir} Let $S$ and $T$ be two arbitrary integral $n$-simplices in $\mathbb{R}^n$. If $S$ and $T$ are unimodularly equivalent, then there exists some $A\in\mathrm{GL}_{n + 1}(\mathbb{Z})$ and $\phi\in \mathrm{Sym}(\sigma_{D_T})$, such that $$AS^* = T^*,$$ where $S^*$ and $T^*$ are defined as in \eqref{eq**}.
\end{proposition}

\begin{proof}
Suppose that $S$ and $T$ are unimodularly equivalent. Then by Definition~\ref{uni eqi}, there exists $B\in \operatorname{GL}_n(\mathbb{Z})$, $\mathbf{c}\in\mathbb{Z}^n$, and some bijection $f:\{1,2,\hdots, n + 1\}\to \{1,2, \hdots, n + 1\}$, such that $$B\mathbf{s_i} + \mathbf{c} = \mathbf{t_{f(i)}},$$ for all $1 \leq i \leq {n + 1}$. Note that this is equivalent to having $$\begin{pmatrix}
B\mathbf{s_1} + \mathbf{c} & B\mathbf{s_2} + \mathbf{c} & \hdots & B\mathbf{s_{n + 1}} + \mathbf{c} \\
1 & 1 & \hdots & 1
\end{pmatrix}
= \begin{pmatrix}
\mathbf{t_{f(1)}} & \mathbf{t_{f(2)}} & \hdots & \mathbf{t_{f(n + 1)}} \\
1 & 1 & \hdots & 1
\end{pmatrix},
$$ which in turn is equivalent to $$\begin{pmatrix} B & \mathbf{c} \\ \mathbf{0} & 1 \end{pmatrix}  \begin{pmatrix}
\mathbf{s_1} & \mathbf{s_2} & \hdots & \mathbf{s_{n + 1}} \\
1 & 1 & \hdots & 1
\end{pmatrix} \\
= \begin{pmatrix}
\mathbf{t_{f(1)}} & \mathbf{t_{f(2)}} & \hdots & \mathbf{t_{f(n + 1)}} \\
1 & 1 & \hdots & 1
\end{pmatrix}.$$ Now let $\phi$ be a bijective function, such that $$\phi\dbinom{\mathbf{t_i}}{1}= \dbinom{\mathbf{t_{f(i)}}}{1},$$ for all $1\le i\le n + 1.$ Clearly, $\phi\in \mathrm{Sym}(\sigma_{D_T}).$ If we let $S^*$ and $T^*$ be defined as in \eqref{eq**}, then setting $$A=\begin{pmatrix} B & \mathbf{c} \\ \mathbf{0} & 1 \end{pmatrix},$$ gives us $A\in\mathrm{GL}_{n + 1}(\mathbb{Z})$ and $AS^*=T^*,$ as desired. 
\end{proof}
This, combined with Remark~\ref{sym} gives us the following proposition: 

\begin{proposition}\label{ordi} Let $S$ and $T$ be two arbitrary integral $n$-simplices in $\mathbb{R}^n$. $S$ and $T$ are unimodularly equivalent if and only if there exists some $A\in\operatorname{GL}_{n+1}(\mathbb{Z})$ and $\phi\in \mathrm{Sym}(\sigma_{D_T})$, such that $$AS^* = T^*,$$ where $S^*$ and $T^*$ are defined as in \eqref{eq**}. \end{proposition}

Using Proposition~\ref{ordi}, we have the following SageMath code (SC) to check whether two arbitrary integral $n$-simplices $S,T\subset\mathbb{R}^n$ are unimodularly equivalent.

\textbf{SageMath code (SC):}
\begin{lstlisting}{language=Python}
sage: M_S = matrix([(s_11,s_21,...,s_n1,1),(s_12,s_22,...,s_n2,1),...,
            (s_1(n+1),s_2(n+1),...,s_n(n+1),1)])
sage: M_T = matrix([(t_11,t_21,...,t_n1,1),(t_12,t_22,...,t_n2,1),...,
            (t_1(n+1),t_2(n+1),...,t_n(n+1),1)])
sage: D_S = M_S.transpose()
sage: D_T = M_T.transpose()
sage: for s in Permutations([1,2,...,n+1]):
        P = s.to_matrix()
        N = (D_T*P)*D_S.inverse()
        b = True
        for i in range(n+1):
            for j in range(n+1):
                if N[i,j].is_integer() == False:
                    b = False
                    break
        if b:
            if (N.determinant() == 1 or N.determinant() == -1):
                print(N)
                print(" ")
                print(D_T*P)
                print("---------------------------------")
sage: print("Done")
\end{lstlisting}
\textbf{Explanation of the code:} Suppose that we are given two arbitrary integral $n$-simplices $S,T\subset\mathbb{R}^n$, and we wish to check the unimodular equivalence of $S$ and $T$ using the code above. 

We begin by entering the elements of the matrix $M_S$ and $M_T$, which are set to be the transposes of $D_S$ and $D_T$ respectively. This is done by executing the lines of code 
\begin{lstlisting}{language=Python}
sage: D_S = M_S.transpose()
sage: D_T = M_T.transpose()
\end{lstlisting}

Next, the following lines of code
\begin{lstlisting}{language=Python}
sage: for s in Permutations([1,2,...,n+1]):
        P = s.to_matrix()
\end{lstlisting}
mark the beginning of the for loop, and set $P$ to be a distinct permutation matrix of order $n$ in each of the $(n+1)!$ iterations. 

Now since $|\mathrm{Sym}(\sigma_{D_T})|=(n+1)!$, we let $$\mathrm{Sym}(\sigma_{D_T})= \{\phi_1,\phi_2,\dots,\phi_{(n+1)!}\}.$$ Also, for each $1\le i\le (n+1)!,$ let $$\Omega_i(D_T)= \begin{pmatrix}\phi_i\dbinom{\mathbf{t_1}}{1} & \phi_i\dbinom{\mathbf{t_2}}{1} & \dots & \phi_i\dbinom{\mathbf{t_{n + 1}}}{1}\end{pmatrix}.$$ Note that $\Omega_i(D_T)$ is computed for all $1 \leq i \leq (n+1)!$ using \texttt{D\_T*P}, as the loop variable \texttt{s} changes its value $(n+1)!$ times.

Then for all $1\le i\le (n+1)!$, if we set $$A_i= \Omega_i(D_T)(D_S)^{-1},$$ then we see that all the values of $A_i$ are similarly generated using \texttt{N = (D\_T*P)*D\_S.inverse()}, as the loop variable \texttt{s} changes its value $(n+1)!$ times.

Now, using Proposition~\ref{ordi}, we can conclude that $S$ and $T$ are unimodularly equivalent if and only if $A_i\in\operatorname{GL}_{n+1}(\mathbb{Z})$ for some $1\le i\le (n+1)!$. This is equivalent to having \texttt{N} to be an integer matrix with determinant $\pm 1$ for some iteration of the code. We check this at every iteration by executing the following lines of code:
\begin{lstlisting}{language=Python}
    b = True
    for i in range(n+1):
        for j in range(n+1):
            if N[i,j].is_integer() == False:
                b = False
                break
\end{lstlisting}
Here \texttt{b} stores \texttt{False} if and only if some element of \texttt{N} is not an integer, thereby discarding the data. It takes $\mathcal{O}(n^2)$ time to check whether all the elements of \texttt{N} are integers. If we get any $A_i$, or equivalently any \texttt{N}, for which all the entries are an integer, then we next check if $\det A_i=\pm 1$ by executing the line of code
\begin{lstlisting}{language=Python}
    if (N.determinant() == 1 or N.determinant() == -1):
\end{lstlisting}
If this evaluates to be true, then $A_i$ or \texttt{N} is printed along with $\Omega_i(D_T)$, by executing the following lines of code:
\begin{lstlisting}{language=Python}
    print(N)
    print(" ")
    print(D_T*P)
\end{lstlisting}
In any other case, the above-mentioned lines of code are not executed; in which case nothing is printed. The program terminates by printing \texttt{Done}. 

Combining Proposition~\ref{fordir} and Proposition~\ref{revdir} gives us the following proposition: 

\begin{proposition}\label{equivol}Let $S$ and $T$ be two arbitrary integral $n$-simplices of equal volume in $\mathbb{R}^n$. Then $S$ and $T$ are unimodularly equivalent if and only if there exists some $A\in\mathbb{M}_{n + 1}(\mathbb{Z})$ and $\phi\in \mathrm{Sym}(\sigma_{D_T})$, such that $$AS^* = T^*,$$ where $S^*$ and $T^*$ are defined as in \eqref{eq**}. \end{proposition}

Since two Ehrhart-equivalent simplices have the same Euclidean volume, when checking the unimodular equivalence of two Ehrhart-equivalent integral $n$-simplices $S,T\subset\mathbb{R}^n$, it suffices to use (SC) with the removal of the following line of code:
\begin{lstlisting}{language=Python}
    if (N.determinant() == 1 or N.determinant() == -1):
\end{lstlisting} 
We will refer to this version as the ``modified SC.''

Using modified SC, we find that all Ehrhart-equivalent integral $4$-simplices in Examples $1-5$ are mutually unimodularly equivalent, and therefore also mutually $\operatorname{GL}_4(\mathbb{Z})$-equidecomposable by Remark~\ref{uni}. Again, using (SC), we show that the integral $4$-polytopes $P_1$ and $P_2$ in Example 6 (which are Ehrhart-equivalent) are $\operatorname{GL}_4(\mathbb{Z})$-equidecomposable. The proof is as follows:

\begin{proof} The integral $4$-polytopes $P_1,P_2\subset\mathbb{R}^4$ have been defined as $$P_1= \mathrm{conv}\{(0,0,0,0),(0,0,0,1),(0,0,1,0),(0,1,0,0),(1,0,0,0),(1,1,1,1)\},$$ and $$P_2=\mathrm{conv}\{(0,0,0,0),(1,2,2,4),(3,2,2,3),(1,1,1,2),(0,3,2,3),(5,8,7,12)\}.$$ Using the \texttt{triangulate()} function of SageMath we find that $P_1$ can be decomposed into the following four relatively open $4$-simplices $$S_1= \mathrm{conv}\{(0, 0, 0, 0), (0, 0, 0, 1), (0, 0, 1, 0), (0, 1, 0, 0), (1, 1, 1, 1)\},$$ $$S_2= \mathrm{conv}\{(0, 0, 0, 0), (0, 0, 0, 1), (0, 0, 1, 0), (1, 0, 0, 0),(1, 1, 1, 1)\},$$ $$S_3= \mathrm{conv}\{(0, 0, 0, 0), (0, 0, 0, 1), (0, 1, 0, 0), (1, 0, 0, 0), (1, 1, 1, 1)\},$$ $$S_4= \mathrm{conv}\{(0, 0, 0, 0), (0, 0, 1, 0), (0, 1, 0, 0), (1, 0, 0, 0), (1, 1, 1, 1)\},$$ and $P_2$ can be decomposed into the following four relatively open $4$-simplices $$T_1= \mathrm{conv}\{(0, 0, 0, 0), (0, 3, 2, 3), (1, 1, 1, 2), (1, 2, 2, 4), (5, 8, 7, 12)]\},$$ $$T_2= \mathrm{conv}\{(0, 0, 0, 0), (0, 3, 2, 3), (1, 1, 1, 2), (3, 2, 2, 3), (5, 8, 7, 12)\},$$ $$T_3= \mathrm{conv}\{(0, 0, 0, 0), (0, 3, 2, 3), (1, 2, 2, 4), (3, 2, 2, 3),(5, 8, 7, 12)\},$$ $$T_4= \mathrm{conv}\{(0, 0, 0, 0),(1, 1, 1, 2), (1, 2, 2, 4), (3, 2, 2, 3), (5, 8, 7, 12)\}.$$ 

Using (SC), we show that $S_i$ is unimodularly equivalent to $T_i$, for all $1\le i\le 4,$ and therefore conclude that $P_1$ is $\operatorname{GL}_4(\mathbb{Z})$-equidecomposable to $P_2.$\end{proof} 

\begin{rem}\label{true} Since all Ehrhart-equivalent integral $4$-simplices in Examples 1-5 are mutually $\operatorname{GL}_4(\mathbb{Z})$-equidecomposable, by Proposition~\ref{dilations}, we can conclude that their sixth dilations are mutually $\operatorname{GL}_4(\mathbb{Z})$-equidecomposable as well. Similarly, since we proved that the Ehrhart-equivalent integral $4$-polytopes $P_1$ and $P_2$ in Example 6 are $\operatorname{GL}_4(\mathbb{Z})$-equidecomposable, $6P_1$ and $6P_2$ are $\operatorname{GL}_4(\mathbb{Z})$-equidecomposable as well. Thus Conjecture~\ref{ked} for the specific case of $n=4$ holds true in each of the above-mentioned cases.\end{rem}

Though the examples above may suggest that any two integral $4$-simplices $S,T \subset \mathbb{R}^4$ that are Ehrhart-equivalent are also unimodularly equivalent, this turns out to be false.
A counterexample is motivated from a $2$-dimensional example described in \cite{3D}, where we have the following rational triangles in $\mathbb{R}^2$:
$$P_1 =  \mathrm{conv}\left \{(-4,0),(-1,0),\left(-3,\tfrac{2}{3}\right)\right \} \quad \text{and} \quad P_2 =  \mathrm{conv}\{(1,0),(3,0),(1,1)\},$$
as shown below in Figure~\ref{fig: image1}. 

\begin{figure}[h]
    \centering
    \includegraphics[width=0.6\textwidth]{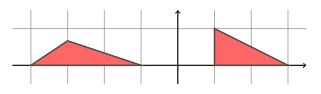}
    \caption{The above figure shows $P_1$ (left) and $P_2$ (right).}
    \label{fig: image1}
\end{figure}

We can decompose $P_1$ into two relatively open $2$-simplices, $$Q_1 =  \mathrm{conv}\left\{(-3,0),(-4,0),\left (-3,\tfrac{2}{3}\right )\right \} \quad \text{and}\quad Q_2 =  \mathrm{conv}\left\{(-3,0),(-1,0),\left(-3,\tfrac{2}{3} \right )\right\},$$ and $P_2$ into two relatively open triangles ($2$-simplices), $$Q_1' =  \mathrm{conv}\left\{(3,0),\left(1,\tfrac{2}{3}\right),(1,1)\right\} \quad\text{and}\quad Q_2' =  \mathrm{conv}\left \{(1,0),(3,0),(1,\tfrac{2}{3})\right\},$$ as shown in Figure~\ref{fig: image2}.
\begin{figure}[h]
    \centering
    \includegraphics[width=0.6\textwidth]{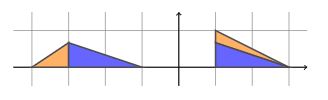}
    \caption{The above figure shows the transformation of $Q_1$ and $Q_2$ (left) to $Q_1'$ and $Q_2'$ (right) respectively. Here, $Q_1$ and $Q_1'$ are filled in orange, and $Q_2$ and $Q_2'$ are filled in blue.} 
    \label{fig: image2}
\end{figure}

Now, note that $Q_2'$ is just $Q_2$ shifted, and $Q_1'$ can be obtained from $Q_1$ via the affine-unimodular transformation
\begin{align*}
    \begin{pmatrix}x \\ y \end{pmatrix} \mapsto \begin{pmatrix} 2 & -3 \\ -1 & 1\end{pmatrix}\begin{pmatrix}x \\ y \end{pmatrix} + \begin{pmatrix} 9 \\ -3\end{pmatrix},
\end{align*}
which implies that $Q_1$ is unimodularly equivalent to $Q_1'$ and $Q_2$ is unimodularly equivalent to $Q_2'$. This, by definition implies that $P_1$ and $P_2$ are $\operatorname{GL}_2(\mathbb{Z})$-equidecomposable, and hence Ehrhart-equivalent.  

To turn this into a working counterexample, we take $P_1$ and $P_2$ and perform the following transformations:
\begin{align*}
    &\begin{pmatrix}x \\ y \end{pmatrix} \mapsto 3\begin{pmatrix}x \\ y \end{pmatrix} + \begin{pmatrix} 12 \\ 0\end{pmatrix} \quad\text{and}\quad \begin{pmatrix}x \\ y \end{pmatrix} \mapsto 3\begin{pmatrix}x \\ y \end{pmatrix} + \begin{pmatrix} -3 \\ 0\end{pmatrix},
\end{align*}
respectively, so that $P_1$ and $P_2$ are first dilated to become integral simplices, and then translated by a lattice vector, so that one of its vertices is shifted to the origin, for convenience. $P_1$ and $P_2$ are then transformed into $P'_1$ and $P'_2$, respectively, where $$P'_1= \mathrm{conv}\{(0,0),(9,0),(3,2)\}$$
and $$P'_2= \mathrm{conv}\{(0,0),(6,0),(0,3)\}.$$

Following these transformations, we construct the $4$-dimensional pyramids of $P_1$ and $P_2$, given by
$$R_1= \mathrm{conv}\{(0,0,0,0),(9,0,0,0),(3,2,0,0),(0,0,1,0),(0,0,0,1)\}$$
and
$$R_2= \mathrm{conv}\{(0,0,0,0),(6,0,0,0),(0,3,0,0),(0,0,1,0),(0,0,0,1)\},$$
respectively. We find that
$$L_{R_1}(t) = \frac{3}{4} t^4 + 4t^3 + \frac{29}{4} t^2 + 5t + 1 = L_{R_2}(t),$$ which implies that $R_1$ and $R_2$ are Ehrhart-equivalent. However, using our algorithm, it turns out that $R_1$ and $R_2$ are not unimodularly equivalent. This therefore serves as a valid counterexample; $R_1$ and $R_2$ are Ehrhart-equivalent integral $4$-polytopes, but they are not unimodularly equivalent. Note, however, that this counterexample does not disprove Conjecture~\ref{ked} for the specific case of $n=4$.

\section{Unimodular equivalence of pyramids of simplices} \label{pyramids} We now consider whether, given two arbitrary unimodularly equivalent $2$-simplices in $\mathbb{R}^2$, their $n$-dimensional pyramids in $\mathbb{R}^n$ are unimodularly equivalent.

\begin{theorem} \label{thm: main}
Let $Q_{2,1}, Q_{2,2}\subset\mathbb{R}^2$ be two arbitrary integral $2$-simplices, such that 
$$Q_{2,1}= \mathrm{conv}\{(0,0),(a,b),(c,d) \},$$ $$Q_{2,2}= \mathrm{conv}\{(0,0),(a',b'),(c',d')\}.$$
Now for any $n>2$, let $Q_{n,1}, Q_{n,2}\subset\mathbb{R}^n$ be the $n$-dimensional pyramids generated by $Q_{2,1}$ and $Q_{2,2}$ respectively by adding the elementary basis vectors $\mathbf{e_3}, \dots, \mathbf{e_n}$; in particular,
\begin{align*}
    &Q_{n,1} =  \mathrm{conv}\{\mathbf{0},(a,b,\dots,0),(c,d,\dots,0),\mathbf{e_3},\dots,\mathbf{e_n}\}, \\
    &Q_{n,2} =  \mathrm{conv}\{\mathbf{0},(a',b',\dots,0),(c',d',\dots,0),\mathbf{e_3},\dots,\mathbf{e_n}\}.
\end{align*}
Then $Q_{2,1}$ and $Q_{2,2}$ are unimodularly equivalent if and only if $Q_{n,1}$ and $Q_{n,2}$ are unimodularly equivalent. 
\end{theorem}
This theorem establishes a direct bijective relationship between the unimodular equivalence of simplices and their pyramids. The proof of Theorem~\ref{thm: main} is long and technical, and so we split up the forward and backward directions into two propositions and prove them separately.

\begin{proposition}\label{higher n}
For all $n \ge 2$, let $Q_{n,1}$, and $Q_{n,2}$ be as defined in Theorem~\ref{thm: main}. If $Q_{2,1}$ and $Q_{2,2}$ are unimodularly equivalent, then for all $n\ge 2$, $Q_{n,1}$ and $Q_{n,2}$ are also unimodularly equivalent, $\mathrm{GL}_n(\mathbb{Z})$-equidecomposable and Ehrhart-equivalent.
\end{proposition}

\begin{proof} Suppose that $Q_{2,1}$ and $Q_{2,2}$ are unimodularly equivalent. 
Then there exists some affine-unimodular transformation
\begin{align*}
    U(\mathbf{v})= M \mathbf{v} + \binom{x}{y}
\end{align*}
defined on $\mathbf{v}\in \mathbb{R}^2$, where $M\in \operatorname{GL}_2(\mathbb{Z})$ and $\dbinom{x}{y}\in\mathbb{Z}^2.$

For any $n > 2$, let
\begin{align*}
    B_n =  \begin{pmatrix} M & A_{n-2} \\ O & I_{n-2}\end{pmatrix},
\end{align*}
where $I_{n-2}$ is the $(n-2)$-dimensional identity matrix, and $A_{n-2}\in\mathbb{M}_{2,n-2}(\mathbb{Z})$, is the matrix given by
\begin{align*}
    A_{n-2} = \begin{pmatrix} -x & \dots &-x \\ -y & \dots & -y \end{pmatrix}.
\end{align*}
We see that $\det(B_n)=\pm \det (M)=\pm 1,$ so $B_n\in\operatorname{GL}_n(\mathbb{Z})$. Thus $Q_{n,1}$ and $Q_{n,2}$ are unimodularly equivalent via the affine-unimodular transformation $V_n\in\mathrm{GL}_n(\mathbb{Z})\ltimes \mathbb{Z}^n$. In particular, $V_n(Q_{n,1})=Q_{n,2}$, where
\begin{align*}
    V_n(\mathbf{v}) = B_n\mathbf{v} + \begin{psmallmatrix}x \\ y \\ 0 \\ \svdots \\ 0\end{psmallmatrix} 
\end{align*}
for all $\mathbf{v}\in \mathbb{R}^n$. Since, $Q_{n,1}$ and $Q_{n,2}$ are integral simplices in $\mathbb{R}^n$, they must also be $\mathrm{GL}_n(\mathbb{Z})$-equidecomposable and Ehrhart-equivalent by Remark~\ref{uni}.
\end{proof}

We now prove Proposition~\ref{backwards dir}, the backwards direction of Theorem~\ref{thm: main}.

\begin{proposition} \label{backwards dir}
For all $n \ge 2$, let $Q_{n,1}$, and $Q_{n,2}$ be as defined in Theorem~\ref{thm: main}. If for some $n>2$, $Q_{n,1}$ and $Q_{n,2}$ are unimodularly equivalent, then $Q_{2,1}$ and $Q_{2,2}$ are also unimodularly equivalent.
\end{proposition}

For Proposition~\ref{backwards dir}, let $U\in \operatorname{GL}_n(\mathbb{Z})\ltimes \mathbb{Z}^n$ be the affine-unimodular transformation such that $U(Q_{n,1}) = Q_{n,2}$. Then there exist some $M = (m_{ij}) \in \operatorname{GL}_n(\mathbb{Z})$ and $\mathbf{u}\in\mathbb{Z}^n$ be such that
\begin{align*}
    U(\mathbf{v}) = M\mathbf{v}+\mathbf{u}, 
\end{align*}
for all $\mathbf{v}\in Q_{n,1}.$

We first require some preparatory lemmas.

\begin{lemma}\label{case 2}
Let $Q_{n,1}$, $Q_{n,2}$ be $n$-dimensional pyramids as defined in Theorem~\ref{thm: main}, and let $U\in \operatorname{GL}_n(\mathbb{Z})\ltimes \mathbb{Z}^n$ be an affine-unimodular transformation such that $U(Q_{n,1}) = Q_{n,2}$. If $$U(a,b,0,\dots,0) = \mathbf{e_i}$$ for some $3\le i\le n$, then
$$
\begin{pmatrix}
    a & \frac{c}{k} \\
    b & \frac{d}{k}
\end{pmatrix}
\in \operatorname{GL}_n(\mathbb{Z}),
$$
where $k=\gcd(c,d)$.
\end{lemma}

\begin{proof}
Since $U(a,b,0,\dots,0) = \mathbf{e_i}$, we have a unimodular matrix $M$ and translation vector $\mathbf{u}$ such that
\begin{align*}
    M\begin{psmallmatrix}a\\b\\ 0\\ \svdots\\ 0\end{psmallmatrix}+\mathbf{u}=\mathbf{e_i}.
\end{align*}
Then, looking at the $i$-th component of both sides of the matrix equation gives us
\begin{align}\label{fora}
    am_{i1}+ bm_{i2}+u_i = 1.
\end{align}

We know that $U(c,d,\dots,0)$ is one of the vertices of $Q_{n,2}$, so it is either $(0,0,0,\dots, 0)$, $(a',b',0,\dots,0)$, $(c',d',0,\dots,0)$, or $\mathbf{e_j}$ for some $j \neq i$. Note that in any of these cases, its $i$-th component is $0$. Thus, we have our second equation,
\begin{align}\label{forb}
    cm_{i1}+dm_{i2}+u_i=0.
\end{align}

Again, since the origin is a vertex of $Q_{n,1}$, it must be mapped to one of the vertices of $Q_{n,2}$, so $\mathbf{u}$ must be a vertex of $Q_{n,2}$. But, since the vertex $(a,c,\dots,0)$ is already mapped to $\mathbf{e_i}$, this means that the origin is not also mapped to $\mathbf{e_i}$. 
This implies that $\mathbf{u}$ must be some other vertex of $Q_{n,2}$. But, all the other vertices of $Q_{n,2}$ have their i-th components equal to zero, so $u_i = 0$. Thus, equations \eqref{fora} and \eqref{forb} reduce to
\begin{align*}
    am_{i1}+bm_{i2}=1, \\
    cm_{i1}+dm_{i2}=0.
\end{align*}
The first equation implies that $$\gcd(m_{i1},m_{i2}) = 1,$$ and the second equation implies that 
\begin{align}\label{dividingms}
\frac{\displaystyle m_{i2}}{\displaystyle m_{i1}}=-\frac{\displaystyle 
c}{\displaystyle d}.\end{align}
Now since $\gcd(m_{i1},m_{i2})=1$, thus from \eqref{dividingms} we can conclude that there exists some $k \in \mathbb{Z}$, such that $c=km_{i2}$ and  $d=-km_{i1}.$

Then $1 = \gcd(m_{i_1},m_{i_2})=\frac{\gcd(c,d)}{|k|}$, where $|k|=\gcd(c,d)$. Thus, we have $$m_{i1}=\pm\frac{\displaystyle d}{\displaystyle \gcd(c,d)}, \text{ and } m_{i2}=\mp\frac{\displaystyle c}{\displaystyle \gcd(c,d)}.$$ 
Substituting this back to the equation $am_{i1}+bm_{i2}=1$, yields $ad-bc=\pm \gcd(c,d).$ 

From here, note that,
\begin{align*}
    \frac{1}{2}\left|ad-bc\right|=\frac{1}{2}\gcd(c,d)=\frac{1}{2}|k|.
\end{align*}
It follows that the matrix
    $\begin{pmatrix}
    a & \frac{c}{k}\\
    b & \frac{d}{k}
    \end{pmatrix}
\in \operatorname{GL}_2{(\mathbb{Z})}.$
\end{proof}

\begin{lemma}\label{case 3} Let $Q_{n,1}$, and $Q_{n,2}$ be $n$-dimensional pyramids as defined above, and let $U\in \operatorname{GL}_n(\mathbb{Z})\ltimes \mathbb{Z}^n$ be an affine-unimodular transformation such that $U(Q_{n,1}) = Q_{n,2}$.
If $U(\mathbf{0}) = \mathbf{e_i},$ for some $3\le i\le n$, then
\[
\begin{pmatrix}
    a & \frac{c-a}{\ell} \\
    b & \frac{d-b}{\ell} \\
\end{pmatrix}
\in \operatorname{GL}_2(\mathbb{Z}).
\]
where $\ell = \gcd{(c-a,d-b)}$
\end{lemma}
\begin{proof}
    Define the translation $T_n(\mathbf{v})= \mathbf{v}-(a,b,0,\dots,0),$ for all $\mathbf{v} \in \mathbb{R}^n$. Then
\begin{align*}
    T_n(Q_{n,1}) = \mathrm{conv}\{(-a,-b,0,\dots,0),\mathbf{0},(c-a,d-b,0,\dots,0),T_n(\mathbf{e_3}),\dots,T_n(\mathbf{e_n})\}.
\end{align*}
Now let $V_n$ be the transformation such that $V_n(\mathbf{v})= U \circ T_n^{-1}(\mathbf{v}),$
for all $\mathbf{v}\in \mathbb{R}^n$. This implies that $V_n(T_n(Q_{n,1}))=U(Q_{n,1})=Q_{n,2}$.
Since $V_n$ is a composition of two affine-unimodular transformations, it must be affine-unimodular. Furthermore, we have $V_n(-a,-b,0,\dots,0) = U_n(\mathbf{0}) = \mathbf{e_i}$. 

Thus we can apply Lemma \ref{case 2} to the $n$-dimensional pyramids $T_n(Q_{n,1})$ and $Q_{n,2}$ along with the affine-unimodular transformation $V_n$. This tells us that the matrix
\[
\begin{pmatrix}
    -a & \frac{c-a}{\ell} \\
    -b & \frac{d-b}{\ell}
\end{pmatrix}
\in \operatorname{GL_n(\mathbb{Z})}
\]
where $\ell = \text{gcd}(c-a,d-b)$.
\end{proof}

Using these lemmas we now prove Proposition~\ref{backwards dir}.

\begin{proof}[Proof of Proposition~\ref{backwards dir}]
Notice that for the simplices $Q_{2,1}$, $Q_{2,2}$, $Q_{n,1}$, $Q_{n,2}$ defined as in the statement of Proposition~\ref{backwards dir}, and $U\in \operatorname{GL}_n(\mathbb{Z})\ltimes \mathbb{Z}^n$ the affine-unimodular transformation such that $U(Q_{n,1}) = Q_{n,2}$, there are exactly $3$ possible cases for the way in which $U$ acts on the individual vertices of $Q_{n,1}$.
\begin{enumerate}
\item $U(\{\mathbf{0},(a,b,0,\dots,0),(c,d,0,\dots,0)\})=\{\mathbf{0},(a',b',0,\dots,0),(c',d',0,\dots,0)\}$.

\item $U(a,b,0,\dots,0) = \mathbf{e_i}$ or $U(c,d,0,\dots,0) = \mathbf{e_i}$ for some $3\le i\le n$. Without loss of generality, assume $U(a,b,0,\dots,0) = \mathbf{e_i}$.
\item $U(0,0,\dots,0)=\mathbf{e_i},$ for some $3\le i\le n$.
\end{enumerate}

\textbf{Case 1:} In this case, $$U(\{(0,0,\dots,0),(a,b,0,\dots,0),(c,d,0,\dots,0)\})=\{(0,0,\dots,0),(a',b',0,\dots,0),(c',d',0,\dots,0)\}.$$ Note that $U(0,\dots,0)$ must have its third through $n$-th components all equal to $0$. 
Letting $M$ and $\mathbf{u}$ be the matrix and translation vector such that
\[
U(\mathbf{v}) = M\mathbf{v} + \mathbf{u},
\]

we have $U(0,\dots,0) = \mathbf{u}$, which implies that $\mathbf{u}$ is of the form $\mathbf{u} = (z,w,0,\dots,0)$ for some $z, w \in \mathbb{Z}$. 

Then we can deduce that $M$ is of the form
\begin{align*}
    M = 
    \begin{pmatrix}
    N & A \\
    O & P
    \end{pmatrix}
\end{align*}
where $O$ is the $(n-2) \times 2$ dimensional zero matrix, $P$ is a $(n-2)$-dimensional permutation matrix, $N$ is some $2\times 2$ integer matrix, and $A$ is the matrix defined by
\begin{align*}
    A =  
    \begin{pmatrix} -z & \dots &-z \\ -w & \dots & -w \end{pmatrix}.
\end{align*}
Since $M \in \operatorname{GL}_n(\mathbb{Z}),$ $\det M=\pm 1$. Now every permutation matrix also has determinant of $\pm 1$, so 
\begin{align*}
    \pm 1 = \det{M} = \det{N}\det{P} = \pm \det{N}\implies \det N=\pm 1.
\end{align*}
This means $N \in \operatorname{GL}_2(\mathbb{Z})$. Thus, note that the transformation $T$, defined by
\begin{align*}
    T(\mathbf{v}) =  N\mathbf{v} + 
    \begin{pmatrix}
    z \\ w
    \end{pmatrix}, 
\end{align*}
for all $\mathbf{v}\in P_1$, is an affine-unimodular transformation from $P_1$ to $P_2$, and so $P_1$ is unimodularly equivalent $P_2$.

\textbf{Cases 2 and 3:}
In this case, $U(a,b,0,\dots,0) = \mathbf{e_i},$ or $U(0,0,0,\dots,0) = \mathbf{e_i},$ for some $3\le i\le n$.

Then consider the affine-linear transformation $U^{-1}$ mapping $Q_{n,2}$ to $Q_{n,1}$. Using a similar analysis to that of $U$, there are $3$ possible cases:
\begin{enumerate}[label = (\alph*)]
\item $U^{-1}(\{(0,0,\dots,0),(a',b',0,\dots,0),(c',d',0,\dots,0)\})=\{(0,0,\dots,0),(a,b,0,\dots,0),(c,d,0,\dots,0)\}$.

\item $U^{-1}(a',b',0,\dots,0) = \mathbf{e_i}$ or $U^{-1}(c',d',0,\dots,0) = \mathbf{e_i}$ for some $3 \le i \le n$. Without loss of generality, $U^{-1}(a',b',0,\dots,0) = \mathbf{e_i}$. 

\item $U^{-1}(0,0,\dots,0)=\mathbf{e_i},$ for some $3\le i\le n$.
\end{enumerate}

\textbf{Case (a):}
Case (a) is impossible since we assumed that $U(a,b,0,\dots,0) = \mathbf{e_i},$ or $U(0,0,0,\dots,0) = \mathbf{e_i},$ for some $3\le i\le n$.

\textbf{Case (b):} Applying Lemma~\ref{case 2} to the inverse map $U^{-1}$ (switching $Q_{n,1}$ and $Q_{n,2}$) we get 
\[
A' = 
\begin{pmatrix}
    a' & \frac{c'}{k'} \\
    b' & \frac{d'}{k'} \\
\end{pmatrix}
\in \operatorname{GL}_n(\mathbb{Z})
\]
where $k' = \gcd{(c',d')}$. Notice that $\text{vol}(Q_{n,2}) = \abs{a'd'-b'c'} = k'$. 

If $U(a,b,0,\dots,0) = \mathbf{e_i},$ then by Lemma~\ref{case 2} we have
\[
A_1 = 
\begin{pmatrix}
    a & \frac{c}{k} \\
    b & \frac{d}{k} \\
\end{pmatrix}
\in \operatorname{GL}_n(\mathbb{Z}).
\]
where $k = \gcd{(c,d)}$. Then $\text{vol}(Q_{n,1}) = \abs{ad-bc} = k$. Note that $\text{vol}(Q_{n,1}) = \text{vol}(Q_{n,2})$ because they are pyramids over $Q_{2,1}$ and $Q_{2,2}$, which have the same volume since they are unimodularly equivalent. This implies that $k = k'$.

Then there is some matrix $T \in \operatorname{GL}_n(\mathbb{Z})$ such that $T A_1 = A'$ since $\operatorname{GL}_n(\mathbb{Z})$ is a group. Right multiplying both sides of the equation by the matrix 
$\begin{pmatrix}
1 & 0 \\
0 & k
\end{pmatrix}$
yields the equation
\[
 T \begin{pmatrix}
    a & \frac{c}{k} \\
    b & \frac{d}{k} \\
\end{pmatrix}
\begin{pmatrix}
1 & 0 \\
0 & k
\end{pmatrix}
=
\begin{pmatrix}
    a' & \frac{c'}{k} \\
    b' & \frac{d'}{k} \\
\end{pmatrix}
\begin{pmatrix}
1 & 0 \\
0 & k
\end{pmatrix}
\]\[
 T \begin{pmatrix}
    a & c \\
    b & d \\
\end{pmatrix}
=
\begin{pmatrix}
    a' & c' \\
    b' & d' \\
\end{pmatrix}
\]
so the simplices $Q_{2,1} = \mathrm{conv}\{(0,0),(a,b),(c,d)\}$ and $Q_{2,2} = \mathrm{conv}\{(0,0),(a',b'),(c',d')\}$ are unimodularly equivalent via multiplication by the matrix $T$.

If instead, $U(0,0,0,\dots,0) = \mathbf{e_i},$ then the argument is roughly the same, except we will have to do a translation at the end. By Lemma~\ref{case 3} we have
\[
A_2 = 
\begin{pmatrix}
    -a & \frac{c-a}{\ell} \\
    -b & \frac{d-b}{\ell} \\
\end{pmatrix}
\in \operatorname{GL}_n(\mathbb{Z}).
\]
where $\ell = \gcd{(c-a,d-b)}$. Notice that the determinant of this matrix is $\frac{ad-bc}{\ell}$. Then $\text{vol}(Q_{n,1}) = \abs{ad-bc} = \ell$, so since $\text{vol}(Q_{n,1}) = \text{vol}(Q_{n,2})$ we have $\ell = k'$.

Then there is some matrix $T \in \operatorname{GL}_n(\mathbb{Z})$ such that $T A_2 = A'$ since $\operatorname{GL}_n(\mathbb{Z})$ is a group. Right multiplying both sides of the equation by the matrix 
$\begin{pmatrix}
1 & 0 \\
0 & \ell
\end{pmatrix}$
yields the equation
\[
 T \begin{pmatrix}
    -a & \frac{c-a}{\ell} \\
    -b & \frac{d-b}{\ell} \\
\end{pmatrix}
\begin{pmatrix}
1 & 0 \\
0 & \ell
\end{pmatrix}
=
\begin{pmatrix}
    a' & \frac{c'}{\ell} \\
    b' & \frac{d'}{\ell} \\
\end{pmatrix}
\begin{pmatrix}
1 & 0 \\
0 & \ell
\end{pmatrix}
\]
\[
 T \begin{pmatrix}
    -a & c-a \\
    -b & d-b \\
\end{pmatrix}
=
\begin{pmatrix}
    a' & c' \\
    b' & d' \\
\end{pmatrix}
\]
So $Q_{2,2} = \mathrm{conv}\{(0,0),(a',b'),(c',d')\}$ is unimodularly equivalent to the simplex $$\mathrm{conv}\{(0,0),(-a,-b),(c-a,d-b)\}$$ via multiplication by the matrix T. However this second simplex is simply a translation of $Q_{1,1}$ by $(a,b)$, so we can conclude that $Q_{2,2}$ and $Q_{2,1}$ are unimodularly equivalent.

\textbf{Case (c)}:
Applying Lemma \ref{case 3}
to the inverse map $U^{-1}$ (switching $Q_{n,1}$ and $Q_{n,2}$) we get 
\[
A' = 
\begin{pmatrix}
    -a' & \frac{c'-a'}{\ell'} \\
    -b' & \frac{d'-b'}{\ell'} \\
\end{pmatrix}
\in \operatorname{GL}_n(\mathbb{Z})
\]
where $\ell' = \gcd{(c'-a',d'-b')}$. Notice that the determinant of this matrix is $\frac{b'c'-a'd'}{\ell}$ and so $\text{vol}(Q_{n,2}) = \abs{a'd'-b'c'} = \ell$. 

If $U(a,b,0,\dots,0) = \mathbf{e_i},$ for some $3\le i\le n$ then this turns out to be exactly analogous to the second part of case 2 (except we switched $Q_{n,1}$ and $Q_{n,2}$).

If $U(0,0,0,\dots,0) = \mathbf{e_i},$ for some $3\le i\le n$, then we can apply similar reasoning to Case 2 to get that there exists $T \in \operatorname{GL}_n(\mathbb{Z})$ such that 
\[
 T \begin{pmatrix}
    -a & \frac{c-a}{\ell} \\
    -b & \frac{d-b}{\ell} \\
\end{pmatrix}
\begin{pmatrix}
1 & 0 \\
0 & \ell
\end{pmatrix}
=
\begin{pmatrix}
    -a' & \frac{c'-a'}{\ell} \\
    -b' & \frac{d'-b'}{\ell} \\
\end{pmatrix}
\begin{pmatrix}
1 & 0 \\
0 & \ell
\end{pmatrix}
\]
\[
 T \begin{pmatrix}
    -a & c-a \\
    -b & d-b \\
\end{pmatrix}
=
\begin{pmatrix}
    -a' & c'-a' \\
    -b' & d'-b' \\
\end{pmatrix}
\]
So the simplices $\mathrm{conv}\{ (0,0),(-a,-b),(c-a,d-b)\}$ and $\mathrm{conv}\{ (0,0),(-a',-b'),(c'-a',d'-b')\}$ are unimodularly equivalent. Translating by the vectors $(a,b)$ and $(a',b')$ respectively therefore implies that $Q_{2,1}$ and $Q_{2,2}$ are unimodularly equivalent.
\end{proof}

\begin{corollary}\label{conj3}
For any $n\ge 2$, there exists two integral $n$-simplices $S_{n,1},S_{n,2}\subset\mathbb{R}^n$, that are not unimodularly equivalent.
\end{corollary}

\begin{proof}
Using (SC), we find that the $2$-simplices $S_{2,1}, S_{2,2} \subset \mathbb{R}^2$, where
$$S_{2,1} = \mathrm{conv}\{(0,0),(9,0),(3,2)\}$$
and
$$S_{2,2} = \mathrm{conv}\{(0,0),(6,0),(0,3)\},$$
are not unimodularly equivalent.
Now for any $n>2$, consider the integral $n$-simplices $S_{n,1}, S_{n,2}\subset\mathbb{R}^n$, defined by $$S_{n,1}= \mathrm{conv}\{\mathbf{0},(9,0,\ldots,0),(3,2,0,\ldots,0),\mathbf{e_3}, \ldots \mathbf{e_n} \}$$
and
$$S_{n,2}= \mathrm{conv}\{\mathbf{0},(6,0,\ldots,0),(0,3,0,\ldots,0),\mathbf{e_3}, \ldots \mathbf{e_n}\}.$$ Note that $S_{n,1}$ and $S_{n,2}$ are the $n$-dimensional pyramids determined by $S_{2,1}$ and $S_{2,2}$ respectively. Thus, by Theorem~\ref{thm: main}, $S_{n,1}$ and $S_{n,2}$ are not unimodularly equivalent.
\end{proof}

\begin{corollary}
Let $Q_{2,1}, Q_{2,2} \subset \mathbb{R}^2$ and $Q_{n,1}, Q_{n,2}$ be the $n$-dimensional pyramids as described in Theorem \ref{thm: main}. If $Q_{2,1}$ and $Q_{2,2}$ are $\operatorname{GL_2}(\mathbb{Z})$-equidecomposable, then $Q_{n,1}$ and $Q_{n,2}$ are also $\operatorname{GL_n}(\mathbb{Z})$-equidecomposable.
\end{corollary}
\begin{proof}
If $Q_{2,1}$ and $Q_{2,2}$ are $\operatorname{GL_n}(\mathbb{Z})$-equidecomposable, then they can be written in the form
$$Q_{2,1}=\bigsqcup_{i=1}^r Q_{i,1}\quad \text{and}\quad Q_{2,2}=\bigsqcup_{i=1}^r Q_{i,2}$$
where each $Q_{i,1}$ and $Q_{i,2}$ are unimodularly equivalent. Thus their $n$-dimensional pyramids are unimodularly equivalent and form partitions of $Q_{n,1}$ and $Q_{n,2}$, meaning $Q_{n,1}$ and $Q_{n,2}$ are also $\operatorname{GL_n}(\mathbb{Z})$-equidecomposable.
\end{proof}

\section{Conclusion and Open Problems} \label{future}

We conclude by gathering a few open questions that arose during the present investigation.

\subsection{Settling Conjecture~\ref{ked}} While Conjecture~\ref{ked} seems to be true for low-dimensional polytopes, we do not yet know whether it is true for dimensions $n\ge 4$. We believe that there is a lot of potential in further investigating Conjecture~\ref{ked}, as it could yield greater insight into the relationship between unimodular equivalence and $\operatorname{GL}_n(\mathbb{Z})$-equidecomposability. Additionally, more examples would be useful to verify whether the conjecture is true or whether a stronger or weaker version could be proven. Generalizing the techniques of Erbe, Haase and Santos \cite{3D} to higher dimensions remains an interesting avenue for future investigation.

\subsection{Generalizations of Theorem \ref{thm: main}} It would also be interesting to investigate further generalizations and implications of Theorem \ref{thm: main}, for example whether a similar result holds when $Q_{n,1}$ and $Q_{n,2}$ are pyramids over $3$-dimensional, or general $m$-dimensional simplices rather than only $2$-dimensional simplices. Additionally, further research could be done investigating other ways to extend Ehrhart-equivalence, unimodular equivalence, and $\operatorname{GL}_n(\mathbb{Z})$-equidecomposability from lower dimensions to their higher dimensional pyramids and vice versa.

\section{Acknowledgements}
This research was sponsored by the Clay Mathematics Institute and conducted during the PROMYS program in 2020. We wish to thank them for the providing us with this research opportunity. We would also like to thank Professor Kiran Kedlaya for his guidance and help throughout the process.

\end{document}